\documentclass[12pt,leqno,oneside]{amsart}

\usepackage{amsmath,amsthm,amssymb,amsfonts,ifpdf,xcomment}

\ifpdf
\usepackage[pdftex,pdfstartview=FitH,pdfborderstyle={/S/B/W 1},%
colorlinks=true, linkcolor=blue, urlcolor=blue, citecolor=blue,%
pagebackref=true]{hyperref}
\else
\usepackage[hypertex]{hyperref}
\fi

\usepackage[abbrev,msc-links,nobysame]{amsrefs}

\newtheorem {thm}   {Theorem}
\newtheorem* {thm*}   {Theorem}
\newtheorem* {prp*}   {Proposition}
\newtheorem {lem}      [thm]    {Lemma}

\newtheorem {prp}[thm]  {Proposition}

\newtheorem {rmk} [thm]    {Remark}
\newcounter{AbcT}

\theoremstyle{definition}

\renewcommand{\a}{\alpha}

\renewcommand{\d}{\delta}

\newcommand{\e}{\varepsilon}

\newcommand{\g}{\gamma}

\renewcommand{\l}{\lambda}

\newcommand{\Z}{{\bf Z}}
\newcommand{\C}{{\bf C}}

\newcommand{\F}{{\bf F}}
\renewcommand{\P}{{\bf P}}

\newcommand {\cG} {{\mathcal G}}

\newcommand{\wt}{\widetilde}

\newcommand{\mat}[4]
{\left(
\begin{array}{cc}
#1 & #2 \\
#3 & #4
\end{array}
\right)}

\DeclareMathOperator{\SL}{SL}
\DeclareMathOperator{\PSL}{PSL}

\DeclareMathOperator{\Tr}{Tr}
\DeclareMathOperator{\Id}{Id}

\DeclareMathOperator{\Girth}{girth}

\title[Expansion of coset graphs]%
{Expansion of coset graphs of $\PSL_2(\F_p)$}
\author{P\'eter P. Varj\'u}
\thanks{
I gratefully acknowledge the support
of the Royal Society.
}


\begin{document}

\begin{abstract}
Let $G$ be a finite group and let $H_1,H_2<G$ be two subgroups.
In this paper, we are concerned with the bipartite graph
whose vertices are $G/H_1\cup G/H_2$ and a coset $g_1H_1$ is connected
with another coset $g_2H_2$ if and only if $g_1H_1\cap g_2 H_2\neq\varnothing$.
The main result of the paper establishes the existence of such graphs with large girth
and large spectral gap. 
Lubotzky, Manning and Wilton use such graphs to construct certain infinite groups of interest
in geometric group theory.
\end{abstract}

\maketitle

\section{Introduction}

Let $G$ be a finite group and let $H_1,H_2<G$ be two subgroups.
In this paper, we are concerned with the bipartite graph $\cG(G;H_1,H_2)$
whose vertices are $G/H_1\cup G/H_2$ and a coset $g_1H_1$ is connected
with another coset $g_2H_2$ if and only if $g_1H_1\cap g_2 H_2\neq\varnothing$.
We write $V(G;H_1,H_2)$ and $E(G;H_1,H_2)$ for the sets of vertices and edges of that
graph.

The purpose of the paper is to show the existence of such graphs with large
girth and large spectral gap.
We introduce the relevant definitions.

We write $A(G;H_1,H_2)$ for the operator acting on $l^2(V(G;H_1,H_2))$
by
\[
A(G;H_1,H_2) f(v)=\frac{1}{|\{u:u\sim v\}|}\sum_{u:u\sim v} f(u),
\]
where $\sim$ is the adjacency relation in $\cG(G;H_1,H_2)$, that is,
we have $u\sim v$ if and only if $(u,v)\in E(G;H_1,H_2)$.
The matrix of $A(G;H_1,H_2)$ in the standard basis is the normalized adjacency matrix
of $\cG(G;H_1,H_2)$.
We write $l^2_0(V(G;H_1,H_2))$ for the space of functions whose average
is $0$ on both $G/H_1$ and $G/H_2$, and write
$A_0(G;H_1,H_2)$ for the restriction of $A(G;H_1,H_2)$ to this subspace.

We write $\Girth(G;H_1,H_2)$ for the length of the shortest non-trivial loop
in $\cG(G;H_1,H_2)$.

The main result of the paper is the following.

\begin{thm}\label{th:main}
There is an absolute constant $c>0$ such that the following holds.
Let $d\in\Z_{\ge3}$.
There is an arbitrarily large finite group $G$ and two subgroups $H_1,H_2<G$
such that the following holds:
\begin{enumerate}
\item both $H_1$ and $H_2$ are cyclic of order $d$, and $|H_1\cap H_2|=1$,
\item $\Girth(G;H_1,H_2)>c\log|G|/\log(d-1)$,
\item $\|A_0(G;H_1,H_2)\|<(d-1)^{-c}$.
\end{enumerate}
\end{thm}

The above result is motivated by a question asked by Henry Wilton to the author
about the existence of such graphs.
Lubotzky, Manning and Wilton use such graphs to construct certain infinite groups of interest
in geometric group theory.
We refer to their paper \cite{LMW-generalized-triangle-groups} for more details.

The constant $c$ in Theorem \ref{th:main} is effective.
It is very easy to see from the proof that $c$ in claim (2) can be taken $1/3-\e$ for any $\e>0$.
The value of $c$ in claim (3) could also be computed, but it would require significantly more effort
and would lead to a very small number.
Those interested in carrying out this task are advised to consult the paper \cite{Kow-Sl2}.

In the proof, the group $G$ will be chosen $\PSL_2(\F_p)$ for a sufficiently large
prime $p$ with $d|(p+1)/2$.
The groups $H_1$ and $H_2$ will be random conjugates of a fixed subgroup
by a one parameter unipotent subgroup.
The girth bound will be obtained by methods similar to those in \cite{GHSSV-girth}, while
the bound on the norm of $A_0(G;H_1,H_2)$ will be proved using the method
of Bourgain and Gamburd \cite{BG-prime}, which builds on the product theorem of
Helfgott \cite{Hel-Sl2}.

\subsection*{Notation}

We write $1$ for the unit element of a multiplicative group, and $\{\pm\Id\}$
for the unit element of $\PSL_2(\F_p)$.

Throughout the paper, we denote by the letters $c,C$ various constants whose value
may change between occurrences.

\subsection*{Organization of the paper}

In Section \ref{sc:girth}, we give estimates for the girth and discuss some slightly
stronger technical results, which will be used in the norm estimates.
In Section \ref{sc:norm} we adapt the method of Bourgain and Gamburd to obtain
the norm estimates.
In (the very short) Section \ref{sc:proof} we combine these results to conclude the
proof of Theorem \ref{th:main}.

\subsection*{Acknowledgments}
The author is grateful to Henry Wilton for asking the question that motivated the
paper.

\section{Girth bounds}\label{sc:girth}

The goal of this section is to prove a technical result, Proposition \ref{pr:girth},
about the existence of subgroups $H_1$, $H_2$
with certain properties in $\PSL_2(\F_p)$.
At the end of the section we will see that if $H_1$ and $H_2$ satisfy the conclusion of this result, then
we can get a good lower bound for $\Girth(\PSL_2(\F_2);H_1,H_2)$.
Some of the properties in this result will be used in the next section, where we estimate the norm of
$A_0(\PSL_2(\F_p);H_1,H_2)$.

Bourgain and Gamburd \cite{BG-prime} gave estimates for the spectral gap of Cayley graphs of $\SL_2(\F_p)$
assuming only a lower bound on the girth.
As we will see in the next section, when we move from $\cG(\PSL_2(\F_p);H_1,H_2)$ to Cayley graphs,
we will need to work with an asymmetric set of generators, that is to say, with a directed Cayley graph.
While this graph has no short directed loop, it does have short undirected loops, and for this reason,
some of the arguments in \cite{BG-prime} do not work verbatim.
This is the reason why we need the full force of Proposition \ref{pr:girth}.

To state the result, we introduce some notation.
We write $F(x_1,\ldots,x_r)$ for the free group of rank $r$ freely generated by $x_1,\ldots, x_r$.
For a group $G$, any element $w\in F(x_1,\ldots,x_r)$ gives rise to a map $w(g_1,\ldots, g_r):G^r\mapsto G$
where we substitute $g_i$ for each generator $x_i$ in $w$ and perform the multiplications in $G$.

\begin{prp}\label{pr:girth}
Let $d\in\Z_{\ge 3}$ be a number and
let
$w\in F(x_1,\ldots, x_{r})$ for some $r\in\Z_{>0}$.
Let $p$ be a sufficiently large (depending on $d$ and $w$)
prime number, which also satisfies $d|(p+1)/2$.

Then there are two subgroups $H_1,H_2<\PSL_2(\F_p)$, such that each is cyclic of order $d$,
$H_1\cap H_2=\{\pm\Id\}$,
and the following holds.

Write $S_0=(H_1\backslash\{\pm\Id\})\cdot (H_2\backslash\{\pm\Id\})$.
Let
\begin{equation}\label{eq:girth}
l\le\frac{\log p}{3 r\log(d-1)}
\footnote{
The constant $3$ in this bound could be improved to $2+\e$ for any $\e>0$.
Moreover, by a simple refinement of the proof, we may achieve that \eqref{eq:no-loop}
has no solution for any $l\le \log p/(2+\e)\log(d-1)$.
See Remark \ref{rm:constant}.
}
\end{equation}
be a positive integer.
Then
\begin{equation}\label{eq:no-loop}
g_1\cdots g_l\not\in H_1
\end{equation}
for any $g_1,\ldots,g_l\in S_0$.
Furthermore, let $\{g_i^{(a)}\}_{i=1,\ldots,l; a=1,\ldots, r} \subset S_0$.
If
\[
w(g_1^{(1)}\cdots g_l^{(1)},\ldots, g_{1}^{(r)}\cdots g_l^{(r)})=\Id,
\]
then there are $1\le a,b\le r$ such that either at least one of $x_a^{-1}x_b$ or $x_b^{-1}x_a$
occurs as a subword of $w$ and
\[
g_{1}^{(a)}=g_{1}^{(b)},\ldots,g_{\lceil l/2\rceil }^{(a)}=g_{\lceil l/2\rceil}^{(b)},
\]
or at least one of $x_ax_b^{-1}$ or $x_bx_a^{-1}$
occurs as a subword of $w$ and
\[
g_{\lfloor l/2\rfloor}^{(a)}=g_{\lfloor l/2\rfloor}^{(b)},\ldots,g_{l}^{(a)}=g_{l}^{(b)}.
\]
\end{prp}

The construction of $H_1$ and $H_2$ is based on the next lemma, whose proof is
similar to arguments in \cite{GHSSV-girth}.
We introduce the notation
\[
u(x)=\pm\mat1x01,
\]
and we write $[a]_{ij}$ for the entry in the $i$'th row and $j$'th column of a matrix $a$.
It is worth remembering that in the case of elements of $\PSL_2(\F_p)$ the matrix entries
are defined up to sign.

\begin{lem}\label{lm:words}
Let $H<\PSL_2(\F_p)$ be a subgroup such that $[h]_{21}\neq 0$
and $\Tr(h)\neq\pm2$
holds for all $h\in H\backslash\{\pm\Id\}$.
Let $f:\F_p\to\PSL_2(\F_p)$ be a map either of the form
\begin{equation}\label{eq:first-form}
f(x)=h_1u(x)h_2u(-x)\cdots h_{2k-1}u(x)h_{2k}u(-x)
\end{equation}
for some $k\in\Z_{>0}$, or of the form
\begin{equation}\label{eq:second-form}
f(x)=h_1u(x)h_2u(-x)\cdots h_{2k-1}u(x)h_{2k}u(-x)h_{2k+1}
\end{equation}
for some $k\in\Z_{\ge 0}$, where
$h_j\in H\backslash\{\pm\Id\}$ are fixed for all $j$.

Then there are at most $4k$ values of $x\in\F_p$ such that
$\Tr(f(x))=\pm2$.
\end{lem}

\begin{proof}
It is clear that $\Tr(f(x))$ is a polynomial of degree at most $2k$, so the
claim will follow if we show that $\Tr(f(x))\not\equiv\pm2$.

We first prove this in the case when $f$ is of the form
\eqref{eq:first-form}.
Multiplying out the product of matrices in \eqref{eq:first-form}, we look for the
coefficient of $x^{2k}$.
In order to get a contribution to this coefficient, we need to choose the entry
$[\cdot]_{12}$ for each factor of the form $u(x)$ or $u(-x)$.
Then the rules of matrix multiplication force us to choose the entry $[\cdot]_{21}$
from each $h_j$ for $j>1$.
Therefore, we see that the coefficient of $x^{2k}$ is $0$ in $[f(x)]_{11}$ and it is
\[
\pm[h_1]_{21}\cdots[h_{2k}]_{21}.
\]
in $[f(x)]_{22}$.
Since $[h_{j}]_{21}$ is never $0$ by the assumptions of the lemma, we found that
$\Tr(f(x))$ is not constant, and this proves the claim.

Now we consider the case when $f$ is of the form
\eqref{eq:second-form} and prove that $\Tr(f(x))\not\equiv\pm2$.
We handle this case by induction on $k$.
If $k=0$, then $\Tr(f(x))=\Tr(h_1)$ is constant but it is not equal to $\pm2$ by the assumptions of the lemma.

Now we suppose that the claim holds for some value of $k$ and prove it for $k+1$.
If $h_1\neq h_{2k+3}^{-1}$, then we write
\[
g(x)=h_{2k+3}h_1u(x)h_2u(-x)\cdots h_{2k-1}u(x)h_{2k+2}u(-x)=h_{2k+3}f(x)h_{2k+3}^{-1}.
\]
We note that $\Tr(f(x))=\Tr(g(x))$ and $g(x)$ is of the form \eqref{eq:first-form}, so the claim follows.
If $h_1= h_{2k+3}^{-1}$, the we write
\[
g(x)=h_2u(-x)\cdots h_{2k-1}u(x)h_{2k+2}=u(-x)h_{2k+3}f(x)h_{2k+3}^{-1}u(x).
\]
We note that $\Tr(f(x))=\Tr(g(x))$ and $g(-x)$ is of the form \eqref{eq:second-form} with the value of $k$
reduced, so the claim follows again.
\end{proof}

\begin{proof}[Proof of Proposition \ref{pr:girth}]
We first recall that $\PSL_2(\F_p)$ has a cyclic subgroup of order $(p+1)/2$ that consists of elements mutually diagonalizable
over $\F_{p^2}$ and the only element in this subgroup that is diagonalizable over $\F_p$ is $\pm\Id$.
Since $d|p+1$, there is a cyclic subgroup $H<\SL_2(\F_p)$ of order $d$ such that $h_{21}\neq 0$ and $\Tr(h)\neq\pm2$
for all $h\in H\backslash\{\pm\Id\}$.
Indeed, if $h_{21}=0$ or $\Tr(h)=\pm2$ for some $h$, then $h$ has eigenvalues in $\F_p$, so it is either diagonalizable in $\F_p$
or it is not diagonalizable in any field extension.

We set $H_1=H$ and $H_2=u(x)Hu(-x)$ and we show that the proposition holds with a suitable
choice of $x\in\F_p$.
By Lemma \ref{lm:words}, for each pairs of elements $h_1,h_2\in H\backslash\{\pm\Id\}$, there are at most $4$ choices of $x$
such that
\[
h_1=u(x)h_2u(-x).
\]
So there are at most $4d^2$ choices of $x\in\F_p$ such that $H_1\cap H_2\neq\{\pm\Id\}$.

Writing
\[
L=\frac{\log p}{3r\log(d-1)},
\]
we let $l\le L$ and let $S_0$ be as in the proposition.
Suppose that
\begin{equation}\label{eq:loop}
g_1\cdots g_l=h_0
\end{equation}
for some $g_1,\ldots, g_l\in S_0$ and $h_0\in H$.
In this case, for each $i$, there are $h_{2i-1},h_{2i}\in H\backslash\{\pm\Id\}$ such that
$g_i=h_{2i-1}u(x)h_{2i}u(-x)$.
It follows from Lemma \ref{lm:words} that for any each choice of $h_0,\ldots,h_{2l}$, there are at most
$4l$ choices of $x$ such that \eqref{eq:loop} holds.
We have at most $L$ choices for $l$, so
for all but $4L^2d(d-1)^{2L}$ choices of $x$, condition \eqref{eq:no-loop} holds.

It remains to verify the last condition and it has a similar proof.
Let $w$ be as in the proposition and let $l\le L$.
For any choice of $\{g_i^{(a)}\}_{i=1,\ldots,l; a=1,\ldots, r}$,
there are elements $h_{2i-1}^{(a)},h_{2i}^{(a)}\in H\backslash\{\pm\Id\}$
such that $g_i^{(a)}=h_{2i-1}^{(a)}u(x)h_{2i}^{(a)}u(-x)$
for each $i$ and $a$ in the relevant ranges.

We say that $\{h_i^{(a)}\}_{i=1,\ldots, 2l;a=1,\ldots,r}$ is admissible if the following holds.
If $1\le a,b\le r$ are such that $x_a^{-1}x_b$ or $x_b^{-1}x_a$ occur in $w$ then
there is $i\le2\lceil l/2\rceil$ such that $h_{i}^{(a)}\neq h_i^{(b)}$.
If $1\le a,b\le r$ are such that $x_ax_b^{-1}$ or $x_bx_a^{-1}$ occur in $w$ then
there is $i>2\lfloor l/2\rfloor$ such that $h_{i}^{(a)}\neq h_i^{(b)}$.

Now it remains to show that 
there are more than $4d^2+4L^2d(d-1)^{2L}$ choices of $x\in \F_p$ such that
\begin{align}\label{eq:word}
w(h_1^{(1)}u(x)&h_2^{(1)}u(-x)\cdots h_{2l-1}^{(1)}u(x)h_{2l}^{(1)}u(-x),\ldots,\nonumber\\
&h_1^{(r)}u(x)h_2^{(r)}u(-x)\cdots h_{2l-1}^{(r)}u(x)h_{2l}^{(r)}u(-x))= \pm\Id
\end{align}
does not hold
for any admissible choices of $\{h_i^{(a)}\}$.

For each admissible choice of $\{h_i^{(a)}\}$,
we can apply Lemma \ref{lm:words} and get that the number of solutions of \eqref{eq:word}
in $\F_p$ is at most $4l|w|$.
The number of choices for $l$ is at most $L$,
and the number of choices for $\{h_i^{(a)}\}$ is at most $(d-1)^{2Lr}$.
Therefore, it remains to verify that
\[
4d^2+4L^2d(d-1)^{2L}+4|w|L^2(d-1)^{2L r}<p,
\]
which certainly holds if $p$ is large enough.
\end{proof}

\begin{rmk}\label{rm:constant}
We could modify the argument as follows.
Instead of requiring that \eqref{eq:word} and \eqref{eq:loop} have no solution for $l\le L$ for the same value of $L$,
we could set
\[
L_1=\frac{\log p}{(2+\e)\log(d-1)}, \quad\text{and}\quad
L_2=\frac{\log p}{(2+\e)r\log(d-1)}
\]
and seek $x$ such that \eqref{eq:loop} has no solution for $l\le L_1$ and \eqref{eq:word} has no solution for
$l\le L_2$.
The above argument than gives that such an $x$ can be found if
\[
4d^2+4L_1^2d(d-1)^{2L_1}+4|w|L_2^2(d-1)^{2L_2r}<p,
\]
which, of course, holds with the above choices of $L_1$ and $L_2$ for any $\e>0$,
provided $p$ is sufficiently large.
This justifies the claim we made after Proposition \ref{pr:girth}.
\end{rmk}

\begin{lem}\label{lm:girth}
Let $G$ be a finite group and let $H_1,H_2<G$ be non-trivial finite subgroups
such that $H_1\cap H_2=\{1\}$.
Write $S_0=(H_1\backslash\{1\})\cdot (H_2\backslash\{1\})$.
Suppose that
\[
g_1\cdots g_l\not\in H_1
\]
for any $l\le L$ and for any $g_1,\ldots,g_l\in S_0$.
Then
\[
\Girth(G;H_1,H_2)>2L.
\]
\end{lem}
\begin{proof}
Let $\Girth(G;H_1,H_2)=2l$.
Consider a loop of length $2l$ in $\cG(G;H_1,H_2)$ and let
\[
\g_0 H_1,\g_1 H_2,\ldots,\g_{2l-1}H_2,\g_{2l} H_1=\g_0 H_1
\]
be the sequence of vertices in that loop. 

Let $i\in[1,2l]$ be an odd number.
Then $\g_{i}H_2\cap\g_{i-1}H_1\neq\varnothing$ implies that $\g_{i-1}^{-1}\g_{i}\in H_1H_2$
so $\g_{i}\in \g_{i-1}H_1H_2$.
If $i\in[1,2l]$ is even, a similar argument gives $\g_{i}\in \g_{i-1}H_2H_1$.
Replacing $\g_1,\ldots, \g_{2l}$ with other representatives of their respective cosets
if necessary, we can ensure that 
for each $i=1,\ldots,2l$, there is $h_i$
such that $\g_{i}=\g_{i-1}h_{i}$, where $h_{i}\in H_1$ if $i$ is odd, and $h_{i}\in H_2$ if $i$ is even.

In order to apply the hypothesis in the lemma, we need to consider whether $h_i=\pm\Id$
is possible for some $i$.
Suppose that $h_i=\pm\Id$ for some even number $i$.
Then
\[
\g_iH_1=\g_{i-2}h_{i-1}H_1=\g_{i-2}H_1,
\]
so the loop has repeated vertices, which is not possible.
This shows that $h_{2i}\neq\pm\Id$ for all $i$.
The same argument is valid for odd numbers $i\ge 3$.

We can conclude therefore that there are $g_2,\ldots,g_l\in S_0$,
$h_1\in H_1$ and $h_2\in H_2\backslash\{\pm\Id\}$ such that
\[
\g_{2l}=\g_0h_1h_2g_2\cdots g_l.
\]
Since $\g_0 H_1=\g_{2l} H_1$,
we conclude that
\[
h_1h_2g_2\cdots g_l\in H_1.
\] 
Replacing $h_1$ with any element of $H_1\backslash\{1\}$ if necessary, we can
find elements $g_1,\ldots, g_l\in S_0$ such that 
\[
g_1g_2\cdots g_l\in H_1.
\]
Using the hypothesis of the lemma, we can conclude from this that $l>L$, which
is precisely what we wanted to prove.
\end{proof}

\section{Norm bounds}\label{sc:norm}

The purpose of this section is to give an estimate for the norm of $A_0(\SL_2(\F_p);H_1,H_2)$
under the assumption that $H_1$ and $H_2$ satisfy the conclusions of Proposition \ref{pr:girth}.
First, we reduce the problem to bounding the norm of averaging operators on (directed)
Cayley graphs.
Then we use the method of Bourgain and Gamburd to give such bounds.
The only difference with \cite{BG-prime} is that we are forced by our setting to work with
an asymmetric  set generators.
This affects only one part of the argument, the estimates for the probability that the random
walk hits a proper subgroup coset.
We will deal with in Lemma \ref{lm:non-concentration} below using the information given to us
by Proposition \ref{pr:girth}.

We introduce some notation.
Let $G$ be a finite group and let $S\subset G$.
We write $A(G;S)$ for the operator acting on $l^2(G)$ by
\[
A(G;S)f(g)=\frac{1}{|S|}\sum_{s\in S}{f(sg)}.
\]
We write $A_0(G;S)$ for the restriction of $A(G;S)$ to $l^2_0(G)$
the orthogonal complement of constant functions in $l^2(G)$.

\begin{lem}\label{lm:coset2Cayley}
Let $G$ be a finite group and let $H_1,H_2<G$ be finite subgroups of the same order
such that $H_1\cap H_2=\{1\}$.
Let $S=H_1H_2$.
Then
\[
\|A_0(G;H_1,H_2)\|^2\le \|A_0(G;S)\|.
\]
\end{lem}

\begin{proof}
Since $A_0(G;H_1,H_2)^2$ is self-adjoint and positive definite, its norm is equal to $\l$, where $\l$ is one of its eigenvalues.
Let $f$ be an eigenfunction of $A_0(G;H_1,H_2)^2$ corresponding to $\l$.
We write $\overline f$ for the restriction of $f$ to $G/H_1$.
Note that $\overline f$ is an eigenfunction of the restriction of $A_0(G;H_1,H_2)^2$ to $G/H_1$
with eigenvalue $\l$.

Consider the function $\wt f: G\to \C$ defined by $\wt f(g)=\overline f(gH_1)$.
We claim that $\wt f$ is an eigenfunction of $A_0(G;S)^T$ with eigenvalue $\l$
and this proves the lemma.

To that end, we observe that by definition $A_0(G;H_1,H_2)^2 \overline f(g_1H_1)$ is the average value
of $\overline f(g_3H_1)$ where $g_3H_1$ is such that there is $g_2\in G$ with
\begin{equation}\label{eq:coset2Cayley}
g_1H_1\cap g_2H_2\neq\varnothing,\quad\text{and}\quad g_2H_2\cap g_3H_1\neq\varnothing.
\end{equation}
The first condition in \eqref{eq:coset2Cayley} holds precisely if
$g_1^{-1}g_2\in H_1H_2$, that is $g_2\in g_1 H_1H_2$.
Similarly, the second condition in \eqref{eq:coset2Cayley} holds precisely if
$g_3\in g_2 H_2H_1$.
Therefore, the two condition in \eqref{eq:coset2Cayley} hold together for some $g_2\in G$ precisely if
$g_3\in g_1 H_1H_2H_1$.
Using that the elements $h_1h_2$ are distinct for $h_1$ running through $H_1$ and $h_2$ running through
$H_2$, (which follows from $H_1\cap H_2=\{1\}$), we conclude that
\[
\l f(gH_1)=A_0(G;H_1,H_2)^2 \overline f(gH_1)=\frac{1}{|S|}\sum_{s\in S}\overline f(gsH_1).
\] 

We observe that
\[
A_0(G;S)^T \wt f(g)=\frac{1}{|S|}\sum_{s\in S} \wt f(gs)=\frac{1}{|S|}\sum_{s\in S} \overline f(gsH_1)
=\l \overline f(g H_1)=\l \wt f(g),
\]
which proves our claim.
\end{proof}

The next result gives the norm estimate for $A_0(\PSL_2(\F_p);S_0)$
it will be applied for the generating set $S_0=(H_1\backslash \{\pm\Id\})(H_2\backslash \{\pm\Id\})$
that we obtain from Proposition \ref{pr:girth}.
The estimate for the norm of $A_0(\PSL_2(\F_p);S)$, where $S=H_1H_2$, can be easily deduced from this
using convexity of the norm.

\begin{prp}\label{pr:SG-Cayley}
For every $\a\in(0,1]$, there is a number $c>0$ such that the following holds.
Let $S_0\subset\PSL_2(\F_p)$ be a finite set with $|S_0|\ge 2$.
Let
\[
w=[[x_1x_2^{-1},x_3x_4^{-1}],[x_5x_6^{-1},x_7x_8^{-1}]]\in F(x_1,\ldots,x_8).
\]
Assume that for all $l\in\Z_{\ge 0}$ with
\[
l\le\a\frac{\log p}{\log |S_0|}
\]
and for all $\{g_i^{(a)}\}_{i=1,\ldots,l;a=1,\ldots,8}\subset S_0$
we have that
\[
w(g_1^{(1)}\cdots g_l^{(1)},\ldots, g_{1}^{(8)}\cdots g_l^{(8)})=\pm\Id,
\]
implies that there are $1\le a\neq b\le 8$ such that
\[
g_{1}^{(a)}=g_{1}^{(b)},\ldots,g_{\lceil l/2\rceil }^{(a)}=g_{\lceil l/2\rceil}^{(b)},
\]
or
\[
g_{\lfloor l/2\rfloor}^{(a)}=g_{\lfloor l/2\rfloor}^{(b)},\ldots,g_{l}^{(a)}=g_{l}^{(b)}.
\]

Then
\[
\|A_0(\PSL_2(\F_p);S_0)\|\le|S_0|^{-c}
\]
provided $p$ is sufficiently large depending on $|S_0|$ and $\a$.
\end{prp}

The next result of Bourgain and Gamburd that we quote from \cite{BG-prime} does all the heavy lifting in the
proof of the norm estimate.
Its proof also encompasses Helfgott's product theorem \cite{Hel-Sl2}.

If $\mu$, $\nu$ are  functions (or measures) on a finite group $G$ we write $\mu*\nu$ for their
convolution, that is
\[
\mu*\nu(g)=\sum_{g_1,g_2:g=g_1g_2}\mu(g_1)\nu(g_2).
\]
We write $\mu^{*l}=\mu*\ldots*\mu$ for the $l$-fold self-convolution of $\mu$.

\begin{prp}[$L^2$-flattening]\label{pr:flattening}
For every $\e>0$, there are $C,\d>0$ such that the following holds.
Let $\mu$ be a probability measure on $\PSL_2(\F_p)$.
Assume that $\mu(gH)<p^{-\e}$ for all $g\in \PSL_2(\F_p)$
and for all proper subgroups $H<\PSL_2(\F_p)$.
Then
\[
\|\mu*\mu\|_2\le C\max(p^{-3/2+\e},p^{-\d}\|\mu\|_2).
\]
\end{prp}

A proof may be found in \cite{BG-prime}*{Proposition 2} or \cite{Kow-Sl2}*{Theorem 4.3}.
The latter reference provides explicit constants and it is also closer to our formulation, though, with
somewhat different notation.
There is one notable difference between the statement in the above references and Proposition \ref{pr:flattening},
which is that we do not assume here that the measure is symmetric, that is we allow $\mu(g)\neq\mu(g^{-1})$.
We note, however, that (at least in our formulation) the assumption of symmetricity is not necessary.
Indeed, it is easy to see that Proposition \ref{pr:flattening} holds for a (not necessarily symmetric) measure
$\mu$ if and only if it holds for the symmetrized measure
\[
\wt\mu(g)=\frac{\mu(g)+\mu(g^{-1})}{2}.
\] 

We write $\chi_S$ for the normalized characteristic function of the set $S$, that is
\[
\chi_S(x)=
\begin{cases}
\frac{1}{|S|} &\text{if $x\in S$},\\
0&\text{otherwise}.
\end{cases}
\]

The purpose of the next lemma is to verify the hypothesis of Proposition \ref{pr:flattening}
for $\mu=\chi_{S_0}^{*l}$ for suitably large $l$.

\begin{lem}\label{lm:non-concentration}
Let $\a>0$ and let $S_0\subset\PSL_2(\F_p)$ be a finite set, which satisfies the assumptions in
Proposition \ref{pr:SG-Cayley}.
Then there is a number $C>0$ depending only on $|S_0|$ such that
\[
\chi_{S_0}^{*l}(gH)\le C p^{-\a/16}
\]
for all
\[
l\ge\a\frac{\log p }{\log|S_0|}
\]
for all $g\in\PSL_2(\F_p)$ and for all
proper subgroups $H<\PSL_2(\F_p)$.
\end{lem}
\begin{proof}
It was observed by Bourgain and Gamburd \cite{BG-prime}*{Proposition 3} that any proper subgroup
$H<\PSL_2(\F_p)$ that has order greater than $60$ satisfies the property that
\[
[[h_1,h_2],[h_3,h_4]]=\pm\Id
\]
for all tuples $h_1,\ldots,h_4\in H$.
This follows easily from the classification of the subgroups of $\PSL_2(\F_p)$
obtained by Dickson, see \cite{Suz-group-I}*{Theorem 6.25}.

We fix an element $g\in\PSL_2(\F_p)$ and a proper subgroup $H<\PSL_2(F_p)$
of order greater than $60$.
Setting
\[
w=[[x_1x_2^{-1},x_3x_4^{-1}],[x_5x_6^{-1},x_7x_8^{-1}]]\in F(x_1,\ldots,x_8),
\]
we see that
\[
w(h_1,\ldots,h_8)=\pm\Id
\]
for all tuples $h_1,\ldots,h_8\in gH$.

We take
\[
l_0=\Big\lfloor\a\frac{\log p}{\log|S_0|}\Big\rfloor
\]
and let $\{X_i^{(a)}\}_{i=1,\ldots,l_0;a=1,\ldots,8}$ be independent random elements of $S_0$
with uniform distribution.
We note that
\[
\P(X_1^{(a)}\cdots X_{l_0}^{(a)}\in gH\text{ for all $a$})=\chi_{S_0}^{*l_0}(gH)^8.
\]
On the other hand, we can write
\begin{align*}
\P(X_1^{(a)}\cdots X_{l_0}^{(a)}&\in gH\text{ for all $a$})\\
\le&\P(w(X_1^{(1)}\cdots X_{l_0}^{(1)},\ldots, X_1^{(8)}\cdots X_{l_0}^{(8)})=\pm\Id)\\
\le&2\cdot{8\choose 2}\cdot \P(X_1^{(1)}=X_1^{(2)},\ldots,X_{\lceil l_0/2\rceil}^{(1)}=X_{\lceil l_0/2\rceil}^{(2)})\\
\le& 56|S_0|^{-l_0/2}.
\end{align*}
Therefore
\[
\chi_{S_0}^{*l_0}(gH)\le 2|S_0|^{-l_0/16}.
\]

We note that the above argument is also valid for $H=\{\pm\Id\}$, so we get that
\[
\chi_{S_0}^{*l_0}(g)\le 2|S_0|^{-l_0/16}
\]
for all $g\in\PSL_2(\F_p)$.
Therefore, the bound
\[
\chi_{S_0}^{*l_0}(gH)\le 60|S_0|^{-l_0/16}.
\]
is true for all proper subgroups $H<\PSL_2(\F_p)$ including those of order not greater than $60$.

Finally, we note for all $l$, we have
\[
\chi_{S_0}^{*(l+1)}(gH)=\frac{1}{|S_0|}\sum_{s\in S_0}\chi_{S_0}^{*l}(sgH)\le\max_{h\in\PSL_2(\F_p)}\chi_{S_0}^{*l}(hH).
\]
By repeated use of this observation, we conclude that
\[
\chi_{S_0}^{*l}(gH)\le 60|S_0|^{-l_0/16}
\]
holds for all $l\ge l_0$.
\end{proof}

\begin{proof}[Proof of Proposition \ref{pr:SG-Cayley}]
We follow Bourgain and Gamburd \cite{BG-prime}.
Let
\[
l_0=\Big\lceil\a\frac{\log p}{\log |S_0|}\Big\rceil.
\]
It follows from Lemma \ref{lm:non-concentration} that we have
\[
\chi_S^{*l}(gH)\le p^{-\a/20}
\]
for all $l\ge l_0$, for all $g\in\PSL_2(\F_p)$ and for all $H<\PSL_2(\F_p)$.

We apply Proposition \ref{pr:flattening} repeatedly for $\mu=\chi_{S_0}^{*2^kl_0}$ with $k=0,1,\ldots$ and $\e=\a/20$
and we conclude that
\[
\|\chi_{S_0}^{*2^kl_0}\|_2\le C\max(p^{-{3/2}+1/20},C^kp^{-k\d}),
\]
for some constants $C,\d>0$ depending only on $\a$.
This means that there is a constant $K$ depending only $\a$ such that
\[
\|\chi_{S_0}^{*Kl_0}\|_2\le C p^{-3/2+1/20}.
\]

We now turn the estimate on $\|\chi_{S_0}^{*Kl_0}\|_2$ into an estimate on
\[
\|A_0(\PSL_2(\F_p);S_0)^{Kl_0}\|
\]
exploiting high multiplicity of eigenvalues based on an idea going back to Sarnak and Xue \cite{SX-multiplicities}.
We write
\begin{align}
\|A_0(\PSL_2(\F_p)&;S_0)^{Kl_0}\|^2\nonumber\\
=&\|(A_0(\PSL_2(\F_p);S_0)^{Kl_0})^TA_0(\PSL_2(\F_p);S_0)^{Kl_0}\|.\label{eq:adjoint}
\end{align}
We note that the eigenspaces of 
the operator
\begin{equation}\label{eq:operator}
(A_0(\PSL_2(\F_p);S_0)^{Kl_0})^TA_0(\PSL_2(\F_p);S_0)^{Kl_0}
\end{equation}
are invariant under the action of $\PSL_2(\F_p)$
from the right, hence it is a sum of non-trivial irreducible representations.
Indeed, the representations are non-trivial, because $l_0^2(\PSL_2(\F_p))$ does not contain constants.
By a result going back to Frobenius \cite{Fro-Sl2}, we know that each non-trivial irreducible representation of $\PSL_2(\F_p)$
is of dimension at least $(p-1)/2$.
It follows that the eigenvalues of \eqref{eq:operator} occur with multiplicity at least $(p-1)/2$.

Since \eqref{eq:operator} is symmetric and positive definite, its norm
equals to its largest eigenvalue.
We have therefore
\begin{align*}
\|\eqref{eq:operator}\|
\le&\frac{\Tr((A_0(\PSL_2(\F_p);S_0)^{Kl_0})^TA_0(\PSL_2(\F_p);S_0)^{Kl_0})}{(p-1)/2}\\
\le&\frac{|\PSL_2(\F_p)|\|\chi_{S_0}^{*Kl_0}\|_2^2}{(p-1)/2}\\
\le& C p^3 p^{-3+1/10}p^{-1}\\
=&Cp^{-9/10}.
\end{align*}
We combine this with \eqref{eq:adjoint} and obtain
\[
\|A_0(\PSL_2(\F_p);S_0)\|^{2Kl_0}\le \|A_0(\PSL_2(\F_p);S_0)^{Kl_0}\|^2\le Cp^{-9/10},
\]
which proves the claim after taking $1/(2Kl_0)$ powers of both sides.
\end{proof}

\section{Proof of Theorem \ref{th:main}}\label{sc:proof}

We let
$G=\PSL_2(\F_p)$, where $p$ is a sufficiently large prime and $d|(p+1)/2$.
We apply Proposition \ref{pr:girth} and let $H_1,H_2<G$ be subgroups that satisfy
the properties claimed in that proposition.
These properties already imply part (1) of the theorem and
by Lemma \ref{lm:girth} they also imply part (2).

To establish part (3) of the theorem, we first apply Lemma \ref{lm:coset2Cayley}
to get
\[
\|A_0(G;H_1,H_2)\|^2\le\|A_0(G;S)\|,
\]
where $S=H_1H_2$.
Letting $S_0=(H_1\backslash\{\pm \Id\})(H_2\backslash\{\pm \Id\})$, we can write
\[
\|A_0(G;S)\|\le\frac{|S_0|}{|S|}\|A_0(G;S_0)\|+\frac{|S|-|S_0|}{|S|}
\]
using convexity of norms.
Finally, we use Proposition \ref{pr:SG-Cayley} with $\a=1/12$
to obtain
\[
\|A_0(G;S_0)\|\le(d-1)^{-c}
\]
for some absolute constant $c>0$.
Combining our estimates, we get
\[
\|A_0(G;H_1,H_2)\|^2<\Big(1-\frac{2d-1}{d^2}\Big)(d-1)^{-c}+\frac{2d-1}{d^2}.
\]
Assuming $c\le 1$ and hence $2(d-1)^{-c}>(2d-1)/d^2$,
both of the bounds
\begin{align*}
\|A_0(G;H_1,H_2)\|^2<&3(d-1)^{-c},\\
\|A_0(G;H_1,H_2)\|^2<&\frac{1}{2}(d-1)^{-c}+\frac{1}{2}
\end{align*}
follow.
If $(d-1)^{-c}<1/10$ we can conclude part (3) of the theorem from the first bound
and in the opposite case, we can conclude it from the second bound.

\bibliography{bibfile}

\bigskip

\noindent{\sc Centre for Mathematical Sciences,
Wilberforce Road, Cambridge CB3 0WA,
UK}\\
{\em e-mail address:} pv270@dpmms.cam.ac.uk

\end{document}